\documentclass[12pt,english]{amsart}

\usepackage{a4wide}
\usepackage[english]{babel}
\usepackage{amsfonts}
\usepackage{amssymb, amsthm}
\usepackage{amsmath}
\usepackage[utf8]{inputenc}
\usepackage{textcomp}

\newtheorem*{theorem*}{Theorem}
\newtheorem{theorem}{Theorem}[section]
\newtheorem{lemma}[theorem]{Lemma}

\newtheorem*{rem*}{Remark}

\theoremstyle{remark}

\DeclareMathOperator*{\clos}{clos}

\DeclareMathOperator*{\re}{Re}

\DeclareMathOperator*{\sgn}{sgn}
\def\leq{\leqslant}
\def\le{\leqslant}
\def\geq{\geqslant}
\def\ge{\geqslant}

\begin{document}

\title{Interpolation for intersections of Hardy-type spaces}

\author{S. V.  Kislyakov \and I. K. Zlotnikov}

\address{St. Petersburg Department of V.A. Steklov \newline
	    Mathematical Institute of the Russian Academy of Sciences \newline
	    27 Fontanka, St. Petersburg 191023, Russia
\newline
	    Chebyshev Laboratory, St. Petersburg State University, \newline
 	    14th Line V.O., 29B, Saint Petersburg 199178 Russia \newline }


\thanks{This research was supported by the Russian Science Foundation (grant No. 18-11-00053)}


\email{skis@pdmi.ras.ru}
\email{zlotnikk@rambler.ru}

\begin{abstract}
Let $(X,\mu)$ be a space with a finite measure $\mu$,  let $A$ and $B$ be $w^*$-closed subalgebras of $L^{\infty}(\mu)$, and let $C$ and $D$ be closed subspaces of $L^p(\mu)$ ($1<p<\infty$) that are modules over $A$ and $B$, respectively. Under certain additional assumptions, the couple
$(C\cap D, C\cap D\cap L^{\infty}(\mu))$ is $K$-closed in $(L^p(\mu), L^{\infty}(\mu))$.

This statement covers, in particular, two cases analyzed previously: that of Hardy spaces on the two-dimensional torus and that of the coinvariant subspaces of the shift operator on the circle. Next, many situations when $A$ and $B$ are $w^*$-Dirichlet algebras also fit in this pattern.
\end{abstract}
\maketitle

\section{Introduction}
Let $(X_0, X_1)$ be a compatible couple of Banach spaces, and let $Y_0, Y_1$ be closed subspaces of $X_0$ and $X_1$, respectively. We remind the reader that the couple $(Y_0, Y_1)$ is said to be $K$-closed in $(X_0, X_1)$ if, whenever
$Y_0+Y_1\ni x=x_0+x_1$ with $x_i\in X_i$, $i=0,1$, we also have $x=y_0+y_1$ with $y_i\in Y_i$ and
$\|y_i\|\leq C\|x_i\|$, $i=0,1$.
  
Obviously, this property implies that the interpolation spaces of the real method for the couple $(Y_0, Y_1)$ are equal to the intersection of the corresponding interpolation spaces for $(X_0, X_1)$ with the sum $Y_0+Y_1$. So, whenever we know the interpolation spaces for the latter couple, $K$-closedness makes interpolation for the former one quite easy.  We recall (see \cite{P} or the survey \cite{K2}) that $K$-closedness does occur in the scale of the Hardy spaces on the unit circle (viewed as subspaces of the corresponding Lebesgue spaces), but now we are interested in the following two more complicated results (in them we assume that $1<p<\infty$, though, in fact, some information beyond this condition is available).
\begin{enumerate}
  \item[i)] The couple $(H^p(\mathbb{T}^2), H^{\infty}(\mathbb{T}^2))$ is $K$-closed in the couple
$(L^p(\mathbb{T}^2), L^{\infty}(\mathbb{T}^2))$ (see \cite{KX}).

  \item [ii)] For an inner function $\theta$ on the unit circle, the couple
 $(H^p\cap\theta\overline{H^p}, H^{\infty}\cap\theta\overline{H^{\infty}})$ is $K$-closed in
 $(L^p(\mathbb{T}), L^{\infty}(\mathbb{T}))$ (see \cite{KZ}).
 \end{enumerate}
It should be noted that an analog of i) in dimensions $n>2$ is an open problem.

Surprisingly, the proofs of Statements i) and ii) have turned out to be quite similar, signalizing that these facts might be particular cases of a more general claim. Such a claim exists indeed and looks roughly like this. Again, here $1<p<\infty$.

\begin{theorem*}
Let $(X,\mu)$ be a space with a finite measure $\mu$,  let $A$ and $B$ be $w^*$-closed subalgebras of $L^{\infty}(\mu)$, and let $C$ and $D$ be closed subspaces of $L^{\infty}(\mu)$ that are Banach modules over $A$ and $B$, respectively. {\rm Under certain additional assumptions}, the couple
$(\clos_{L^p(\mu)}C\cap\clos_{L^p(\mu)} D, C\cap D)$ is $K$-closed in $(L^p(\mu), L^{\infty}(\mu))$.
\end{theorem*}

The \textit{additional assumptions} will be described below and, among other things, will require that some analogs of  the harmonic conjugation operator relative to the algebras $A$ and $B$ have the usual properties, as is, for instance, in the case of $w^*$-Dirichlet algebras. However, the condition for $A$ and $B$ to be  $w^*$-Dirichlet is too restrictive (in particular, we do not insist that a multiple of $\mu$ represent some multiplicative linear functional on either $A$ or $B$). Also, note that the proofs of Statements i) and ii) known previously were based upon the fact that, in those settings, the corresponding harmonic conjugations (or Riesz projections) were classical singular integral operators. In particular, the two proofs  employed Calder{\'o}n--Zygmund decomposition, which is not available in the generality adopted in the theorem. In our presentation, this singular integral operator techniques will be replaced by the use of ``cut-off functions'' belonging to $A$ and $B$. Next, some assumptions about the mutual position of the annihilators of $C$ and $D$ will be made (in the context of the above Statements i) and ii), these assumptions are satisfied trivially).

\section{Description of the assumptions and some examples}
Among the restrictions mentioned above, only the last one will bind the pairs $(A,C)$ and $(B,D)$ together, and the other pertain to these objects taken separately. We discuss these ``individual'' conditions first. Let $\mathcal{E}$ be a $w^*$-closed subalgebra of $L^{\infty}(\mu)$ containing constants (as before, $\mu$ is a finite measure), and let $p$ be as in the theorem. Usually, the algebras we consider will be subject to the following requirement.

{\bf Condition ($\alpha_p$)}. For every nonnegative function $u\in L^p(X,\mu)$, there exists a sequence $\{w_n\}_0^{\infty}$ of functions belonging to $\mathcal{E}$ and such that
\begin{align*}
&\re w_n \ge 0,       \\
&\re w_n\to u \text{ weakly in } L^p(\mu),\text{ and}\\
&\|w_n\|_{L^p} \le C \|u\|_{L^p},  
\end{align*}
with a constant $C$ that may depend only on $p$. 

Observe that, passing to a subsequence, we may assume that the $w_n$ (and not merely their real parts) converge weakly in $L^p(\mu)$, hence some convex combinations of them converge strongly. Another passage to a subsequense ensures also convergence a.e. Clearly, the functions of the last subsequence still have positive real parts that converge strongly and a.e. to $u$. So, we can always strengthen Condition $(\alpha_p)$ in this way for free.

Examples will be discussed later, but now we signalize that this condition can easily be verified in the prototypical case of $\mathcal{E}=H^{\infty}(\mathbb{T})$: it suffices to put 
\begin{equation}\label{trifle}
w_n=u\ast K_n + i\widetilde{u\ast K_n},
\end{equation}
where the $K_n$ are the Fej{\'e}r kernels and tilde indicates harmonic conjugation. 

The next condition will be imposed on modules over algebras in question. Let $\mathcal{F}$ be a $w^*$-closed subspace of $L^{\infty}(\mu)$ that is a module over an algebra $\mathcal{E}$ as above. Throughout the paper, we use the sesquilinear duality $\langle f, g \rangle = \int f \overline{g}.$ Consider the annihilator of $\mathcal{F}$ in $L^1(\mu)$,
\[
\mathcal{F}^{\perp}=\{h\in L^1(\mu)\colon\langle f,h\rangle=0\text{ for all }f\in\mathcal{F}\}.
\]
Denoting the exponent conjugate to $p$ by $q$ and putting $\mathcal{F}^{\perp,q}=L^q(\mu)\cap\mathcal{F}^{\perp}$, we readily see that $\mathcal{F}^{\perp,q}=\{f\in L^q(\mu)\colon \langle f,g\rangle=0\text{ for all }g\in\mathcal{F}_p\}$, where $\mathcal{F}_p$ stands for the $L^p(\mu)$-closure of $\mathcal{F}$.
Next, an easy separation argument shows that $\mathcal{F}^{\perp,q}$ is norm dense in $\mathcal{F}^{\perp}$. The last important observation is that $\mathcal{F}_p$ is a Banach module over $\mathcal{E}$, and $\mathcal{F}^{\perp}$ and $\mathcal{F}^{\perp,q}$ are Banach modules over the algebra $\overline{\mathcal{E}}$.

We are ready to formulate another important condition, which imitates this time some continuity properties of the orthogonal projection of $L^2(\mathbb{T})$ onto $H^2(\mathbb{T})$ (the Riesz projection).

{\bf Condition ($\beta_p$)}. Let $\mathcal{E},\mathcal{F}, p$ and $q$ be as above. We assume that there is a bounded projection $P$ of $L^q(\mu)$ onto $\mathcal{F}^{\perp,q}$ that is also of weak type $(1,1)$:
\begin{equation}\label{weaktype}
    \mu\{x \in X : |Pf(x)| > \lambda \} \le C \frac{\|f\|_{L^1}}{\lambda}
\end{equation}
for all $f\in L^p(\mu)$ and all $\lambda > 0$. As usual, the operator $P$ then extends by continuity to all $f\in L^1(\mu)$ with preservation of the last estimate.

Finally, we state the main result of the paper accurately. Let $A$, $B$, $C$, and $D$ be the objects mentioned in the ``rough'' statement in the Introduction, and let $p$ and $q$ be two conjugate exponents strictly between $1$ and $\infty$.  The spaces $C^{\perp}$, $C^{\perp,q}$ etc. (and all similar spaces for $D$ in place of $C$) are introduced in accordance with the above discussion.

\begin{theorem}\label{main} Suppose that the algebras $A$ and $B$ satisfy Condition $(\alpha_p)$ and that Condition $(\beta_p)$ is fulfilled for the module $C$. Suppose also that the projection $P$ occurring in Condition $(\beta_p)$ has the following property{\rm:}
$$
  P(D^{\perp,q}) \subset D^{\perp,q}. \leqno (\gamma_p)
$$
Then the couple $(C_p\cap D_p, C\cap D)$ is $K$-closed in $(L^p(\mu), L^{\infty}(\mu)).$
\end{theorem}
Condition ($\gamma_p$) is precisely the ``coupling'' condition already promised. 

\subsection{The known cases}\label{known} The two results mentioned in the Introduction fit in this pattern.
\subsubsection{Coinvariant subspaces of the shift operator}
On the unit circle with Lebesgue measure, consider the algebras $A=H^{\infty}(\mathbb{T})$ and $B=\overline{A}$. Next, let $C=A$ and $D=\theta B$, where $\theta$ is an \textit{inner function}, i.e., a function belonging to $A$ whose boundary values have modulus $1$ a.e. These objects are easily seen to satisfy conditions $(\alpha_p)$, $(\beta_p)$, and $(\gamma_p)$ for every $p\in (1,\infty)$, which implies Statement ii) in the Introduction. Indeed, see the paragraph after the definition of Condition  $(\alpha_p)$ to see that this condition is fulfilled for $H^{\infty}(\mathbb{T})$ and for
$\overline{H^{\infty}(\mathbb{T})}$. As to the annihilators, we have $C^{\perp}=\bar{z}\overline{H^1(\mathbb{T})}$ and  $D^{\perp}=\theta zH^1(\mathbb{T})$ ($z$ stands for the identity mapping on the circle).  Condition $(\beta_p)$ is satisfied for $C^{\perp}_q$ because of the standard properties of the ``negative'' Riesz projection
\[
Pf(z)=\sum\limits_{j\leq -1}\hat{f}(j)z^j.
\]
Next, clearly, $P(D^{\perp}_q)=\{0\}\subset D^{\perp}_q$, ensuring also $(\gamma_p)$.
\subsubsection{Two-dimensional torus}
In the situation of Statement i) in the Introduction, we take for $A$ the subalgebra of $L^{\infty}(\mathbb{T}^2)$ consisting of all functions that belong to $H^{\infty}(\mathbb{T})$ in the first variable for almost every value of the second, and for $B$ the same but with the roles of the variables interchanged. Next, we take $C=A$ and $D=B$. Again, it is easily seen that Theorem \ref{main} is applicable, yielding Statement i).

Indeed, now $C^{\perp}$ is the subspace of $L^1(\mathbb{T}^2)$ that consists of all functions lying in $\bar{z}\overline{H^1(\mathbb{T})}$ in the first variable for almost every value of the second, and  $D^{\perp}$ is the same but with the roles of the variables interchanged. To verify $(\alpha_p)$ for both $A$ and $B$, we simply do the easy harmonic conjugation construction described above (see formula \eqref{trifle}), but this time in one variable for every fixed value of the other. For the projection $P$ related to $C$ (or, rather, to $C^{\perp}$), we take the orthogonal projection of $L^2(\mathbb{T}^2)$ onto $C^{\perp,2}$ (i.e., we again simply apply the one-dimensional ``negative Riesz projection'' in the first variable). Note that we do have $P(D^{\perp,q})\subset D^{\perp,q}$ (i.e., condition $(\gamma_p)$ is true), but $P(D^{\perp,q})\ne\{0\}$ this time.

\subsection{$w^*$-Dirichlet algebras and beyond} We remind the reader that a $w^*$-Dirichlet algebra is a $w^*$-closed subalgebra $\mathcal{F}$ of $L^{\infty}(\mu)$ that contains the constants and satisfies the following conditions: (a) $\re\mathcal{F}$ is weak-star dense in $L^{\infty}_{\mathbb{R}}(\mu)$; (b) the measure $\mu$ represents a multiplicative linear functional on $\mathcal{F}$. Thus, $\mu$ must be a probability measure.

We refer to \cite{D, G, HR, SW} for the basic information about $w^*$-Dirichlet algebras. In particular, a discussion of the facts listed below can be found in these sources. It should be noted that originally the term referred to an arbitrary uniform algebra whose $w^*$-closure in $L^{\infty}(\mu)$ is $w^*$-Dirichlet in our sense. However, the entire action develops in this $w^*$-closure in any case, so we prefer to simplify the terminology.

Let $H^p$ stand for the $L^p(\mu)$-closure of  $\mathcal{F}$ (earlier, we denoted the same object by $\mathcal{F}_p$), then the orthogonal projection of $L^2(\mu)$  onto $H^2$ (to be called the \textit{Riesz projection}) is bounded on $L^p(\mu)$ for every $p\in (1,\infty)$ and is of weak type $(1,1)$ (i.e., satisfies \eqref{weaktype}). Next, for every real square integrable function $u$ there is a unique real square integrable function $v$ such that $\int_X vd\mu=0$ and $u+iv\in H^2$. The mapping $u\mapsto v$ is called the \textit{harmonic conjugation operator}. In fact, it maps $L^p(\mu)$ boundedly onto $H^p$ for every  $p\in (1,\infty)$ and is of weak type $(1,1)$. Surely, these properties are equivalent to the corresponding properties of the Riesz projection.

However obvious is the proof of Condition $(\alpha_p)$ in the simplest cases (see formula \eqref{trifle}),  besides the $L^p$-boundedness of harmonic conjugation it involves also approximation of a \textit{positive} function by a.e. bounded \textit{positive} functions whose harmonic conjugates are also bounded a.e. Fortunately, in the case of $w^*$-Dirichlet algebras this approximation condition is fulfilled automatically. We shall prove an even stronger statement.

Let $\mathcal{G}$ be a $w^*$-closed linear subspace of $L^{\infty}(\mu)$, where $\mu$ is a finite measure. We suppose  that $G=\re\mathcal{G}$ is norm-dense in the space of real-valued $\mu$-integrable functions. Next, assume that there exists a linear operator $H$ defined on $\re\mathcal{G}$ and such that
$u+iHu\in\mathcal{G}$
for every $u\in\re\mathcal{G}$. (Note that, again, $H$ is a sort of ``harmonic conjugation'' but now we do not insist a real function $v$ with $u+iv\in\mathcal{G}$ be unique up to a constant summand.) Moreover, we assume that for some $p\in (1,\infty)$ the operator $H$ is bounded on $L^p(\mu)$ and its adjoint $H^*$ is of weak type (1,1). 

\begin{lemma}\label{abstract} Under the above assumptions, $\mathcal{G}$ satisfies Condition $(\alpha_p)$.
\end{lemma}

Note that here $\mathcal{G}$ is not necessarily an algebra, but Condition $(\alpha_p)$ makes sense also for linear subspaces of $L^{\infty}(\mu)$. Thus, Lemma \ref{abstract} tells us that sometimes a condition like $(\beta_p)$ implies $(\alpha_p)$. We postpone the proof of the lemma till the end of this section, and now we give more examples.

\subsubsection{$w^*$-Limits of polynomials} Let $K$ be a compact subset of the complex plane with connected complement, let $X=\partial K$, and let $\mu$ be a probability measure on $X$ representing some point in the interior of $K$. Then the $w^*$-closure of 
all analytic polynomials in $L^{\infty}(\mu)$ is a $w^*$-Dirichlet algebra, see, e.g., \cite{G}.

\subsubsection{The algebra generated by a semigroup} Another classical example of a $w^*$-Dirichlet algebra is the algebra of all functions $f\in L^{\infty}(\mathbb{T}^2)$ with $\hat{f}(k,l)=0$ whenever $k+\omega l<0$ ($\omega$ being an irrational number). Here for $\mu$ we take the normalized Lebesgue measure on the $2$-dimensional torus. However, if $\omega$ is rational, this measure is not multiplicative on the corresponding algebra, a harmonic conjugate of a real function is not unique up to a constant summand, etc. In particular, this situation occurs for the algebra of functions analytic in one variable, though, as we have seen, our ``axioms'' for this algebra can easily be verified directly.

\subsubsection{Weights} Another case where we do not deal with a $w^*$-Dirichlet algebra is that of weighted measures. For example, on the circle there are nontrivial weights such the required continuity properties of the Riesz projection and harmonic conjugation operator remain true in the weighted norms. At the same time, the Lebesgue measure with a weight is quite rarely multiplicative on the corresponding algebra $H^{\infty}$. In this paper, the weights are mentioned only to indicate a source of examples, without a more serious study. Some additional information on weighted measures in the present context can be found in \cite{KZ, K1}.

\subsubsection{Modules over a $w^*$-Dirichlet algebra} Returning to a $w^*$-Dirichlet algebra $\mathcal{G}$, we recall that we also need modules over it as a ``raw material'' for Theorem \ref{main}. All such modules are known, see Theorem 2.2.1 in \cite{SW} and  Theorems 6.1 and 6.2 in Chapter V of \cite{G}. Basically, the interesting examples are of the form $F\mathcal{G}$, where $F$ is a measurable function with $|F|=1$ a.e. We return to the beginning of Subsection \ref{known} and, on the unit circle with normalized Lebesgue measure, consider the algebras $A=H^{\infty}(\mathbb{T})$ and $B=\overline{A}$. Next, we let $C=A$ and $D=FB$, where this time $F$ is a  unimodular function that is not necessarily analytic. It is easy to realize that an arbitrary ``interesting'' couple of modules over $A$ and $B$ (respectively) can be reduced to this one.

Now, if $F$ is inner, i.e., belongs to $A$, we recover the situation already discussed. However, we do not know if the quad $\{A, B, C, D,\}$ satisfies our ``axioms'' in the case of a general $F$ and, apparently, the question is difficult. By way of example, we only indicate a setting in which it is resolved in the positive, at least sometimes. Specifically, let $F=\theta\Phi/\overline\Phi$, where $\Phi$ is an outer function (see, e.g., \cite{Hof} for the definition) and $\theta$ is a ``genuine'' inner function. Clearly, $(\alpha_p)$ is not interesting in this situation, so we must think only about Conditions $(\beta_p)$ and $(\gamma_p)$. We have  $C^{\perp}=\bar{z}\overline{H^1(\mathbb{T})}$ and  $D^{\perp}=F zH^1(\mathbb{T})$, and it is clear that the ``bordered'' negative Riesz projection $v\mapsto\overline{\Phi}^{-1}P(v\overline{\Phi})$ should satisfy Condition $(\gamma_p)$, provided $(\beta_p)$ is fulfilled for it. The latter again reduces to certain weighted estimates for the usual Riesz projection. Apparently, some weighted results proved in \cite{KZ} can be adapted to give some information in the case in question.

\subsubsection{Algebras on product spaces} The last lines show, in particular, that ensuring Condition $(\gamma_p)$ may present a problem. However, there is a class of examples where this is easy. Namely, we can imitate the ``tensor product'' pattern of Statement i) in the Introduction, replacing the spaces $H^{\infty}$ in one variable by (say) two $w^*$-Dirichlet algebras on some measure spaces. Surely, some modules over them (again in one variable each) can be incorporated in an obvious way.

\subsection{Proof of Lemma \ref{abstract}} This is done by duality (a careful application of a separation theorem), but this will remain at the background, because we prefer to use a result already existing, see \cite{corr}. We endow the space $Y=\re\mathcal{G}$ with the norm $\|u\|_Y = \|u\|_{\infty}+\|Hu\|_{\infty}$, where $H$ is the operator introduced before the statement of the lemma. For any real function in $L^2(\mu)$, define a continuous linear functional $\Phi_h$ on $Y$ by the formula $\Phi_h(u)=\int_Xu\bar hd\mu$. If we prove the estimate
\begin{equation}\label{lorentz}
\mu\{|h|>\lambda\}\leq\lambda^{-1}\|\Phi_h\|_{Y^*},
\end{equation}
we are done.

Indeed, then by the main result of the \cite{corr}, for every $g\in Y$ with $L^{\infty}$-norm at most one and every $\varepsilon\in (0,1)$, there exists a measurable function $\varphi$ with values in $[0,1]$ such that  $\mu\{\varphi\ne g\}\leq\varepsilon$ and $\|\varphi g\|_Y\leq C(1+\log\varepsilon^{-1})$. Now, we take a nonnegative function $u\in L^p(\mu)$ to be approximated as in Condition $(\alpha_p)$. By rescaling and truncation, me may assume that it is approximated by a nonnegative function $g$ not exceeding $1$ a.e. within a prescribed precision, and moreover, $\|g\|_{L^p(\mu)}\leq\|u\|_{L^p(\mu)}$. Taking the above $\varphi$ for this $g$, we see that \textit{$\varphi g$ is still nonnegative} and can be made as close to $g$ in $L^p$ as we wish, $H(\varphi g)$ is a bounded function, and, finally, $\|H(\varphi g)\|_{L^p}\leq C\|\varphi g\|_{L^p}\leq C\|u\|_{L^p}$.

To prove inequality \eqref{lorentz}, it suffices to use the definition of the norm in $Y$ to represent $\Phi_h$ in the form
\[
\Phi_h(u)=\int \limits_X u \bar{ A} d\mu +\int \limits_X H(u) \bar{B} d\mu
\]
with some functions $A$ and $B$ satisfying $\|A\|_{L^1} + \|B\|_{L^1}\leq 2\|\Phi_h\|_{Y^*}$. Next, a slight modification of all functions involved allows us to assume that $A,B\in L^2(\mu)$ (no control of the $L^2$-norm is assumed), after which it is safe to write $h=A+H^*B$. So, the desired estimate follows form the weak type $(1,1)$ inequality for $H^*$.

\section{Proof of the main theorem}
\subsection{Cut-off functions}\label{cutoff} We start with a technical lemma to be used quite substantially in the proof. Let, as before, $\mathcal{E}$ be a $w^*$-closed subalgebra of $L^{\infty}(\mu)$ containing the constants and satisfying Condition $(\alpha_p)$. Let also $\gamma$ be a fixed positive integer (for definiteness, we assume that $\gamma\geq 2$). 

\begin{lemma}\label{outFunction}
Suppose that $\varphi\in L^{p}(\mu)$ and $\varphi \ge 1$. Then there exists a function $\Phi \in\mathcal{E}$ with the following properties{\rm:}
\begin{align*}
&\text{\rm(O1)} \quad \|1-\Phi\|_{L^p} \le C_{p,\gamma} \|1- \varphi\|_{L^p},       \\
&\text{\rm(O2)} \quad |\Phi| \le \frac{1}{|\varphi|^{\gamma}}.
\end{align*}
\end{lemma}

\begin{proof}
We put $u = \varphi -1$, then  $u \ge 0$ and, applying Property $(\alpha_p)$ and the observation after it, we can find $w_n\in\mathcal{E}$ that have nonnegative real parts and converge in $L^p(\mu)$ and a.e. to some function $v$ with $\re v=u$. Moreover, we have $\|w_n\|_{L^p}\leq C\|u\|_{L^p}$ with $C$ independent of $u$.
% whose real parts are nonnegative and that converge weakly to $u$ in $L^p(\mu)$, and such that
%$\|w_n\|_{L^q}\leq C\|u\|_{L^p}$. Passing to a subsequence, we may assume that the $w_n$ converge weakly in $L^p(\mu)$ (say, to $v\in L^p(\mu)$), then some convex combinations of these functions converge in norm in $L^p(\mu)$, and, again passing to a subsequence, we may ensure convergence a.e. to $v$. Surely, these operations do not violate the initial properties of the sequence $w_n$.

Now, the functions $1+w_n$ are invertible in $\mathcal{E}$ by the Gelfand theory. Indeed, every nonzero multiplicative linear functional on $\mathcal{E}$ is representable by a probability measure on the maximal ideal space for $L^{\infty}(\mu)$, whence we see that the spectrum of $1 + w_n$ lies in the half-plane $\{\re z\geq 1\}$, hence does not contain the point $0$. Clearly, $\|(1+w_n)^{-1}\|_{\mathcal{E}}=\|(1+w_n)^{-1}\|_{L^{\infty}(\mu)}\leq 1$.

Thus, the function $\Phi_n$ defined by the formula
\begin{equation*}
 \Phi_n = \frac{1}{(1+w_n)^{\gamma}},
\end{equation*}
belongs to $\mathcal{E}$ and has norm at most $1$ in this algebra or in $L^{\infty}(\mu)$. Next, we have
\begin{equation*}
 |1 - \Phi_n| \le C_{\gamma} \left|1 - \frac{1}{1 + w_n}\right| \le C_{\gamma} \left|w_n\right|,
\end{equation*}
whence we deduce that
\begin{equation}\label{Phi_n_1}
 \|1-\Phi_n\|_{L^p} \le C_{p,\gamma} \|1- \varphi\|_{L^p}.
\end{equation}
Clearly, we also have the pointwise estimate
\begin{equation}\label{Phi_n_2}
 |\Phi_n| \le \frac{1}{|1 + \re w_n|^{\gamma}}.
\end{equation}
Now, the choice of the $w_n$ shows that the functions  $\Phi_n$ converge a.e. to some function $\Phi$, which belongs to the unit ball of $\mathcal{E}$ because the convergence is also in the $w^*$-topology of $L^{\infty}(\mu)$ (and, by the way, also in $L^p(\mu)$).

Since $\re w_n\to u$ a.e., the limit passage in \eqref{Phi_n_2} yields (O2). At the same time, (O1) follows by the limit passage in \eqref{Phi_n_1}.
\end{proof}
We give immediately an application of this construction. The lemma presented below will be used as a substitute for a procedure employed originally in the proofs of Statements i) and ii) (see the Introduction) and based on Calder{\'o}n--Zygmund decomposition. To make the claim closer to the setting in which it will be applied, we consider an algebra $\mathcal{E}$ as above, a subspace $\mathcal{F}$ of $L^{\infty}(\mu)$ that is a module over $\mathcal{E}$, and a number $p\in (1,\infty)$. We denote by $q$ the exponent conjugate to $p$, and  use the spaces $\mathcal{F}^{\perp}$ and $\mathcal{F}^{\perp,q}$ introduced before the definition of Property $(\beta_p)$. We also need the closure $\mathcal{Q}$ of $\mathcal{F}^{\perp}$ in the Lorentz space $L^{1,\infty}(\mu)$.
Recall that, clearly, $\mathcal{F}^{\perp}$, $\mathcal{F}^{\perp,q}$, and $\mathcal{Q}$ are modules over $\overline{\mathcal{E}}$.

\begin{lemma}\label{split}
Suppose that $f\in\mathcal{Q}$ and $\lambda>0$. Then there exist functions  $a\in\mathcal{F}^{\perp,q}$ and $b\in\mathcal{Q}$, as well as a set $E$ such that $f = a + b$ and
\begin{gather}
\|a\|_{\mathcal{F}^{\perp,q}} \le C  \lambda^{1/p} \|f\|^{1/q}_{\mathcal{Q}},\label{U1}\\
\int_{X\setminus E} |b| d\mu  \le C \|f\|_{\mathcal{Q}},\label{U2}\\
\mu(E) \le C \lambda^{-1}\|f\|_{\mathcal{Q}},\label{U3}\\
\mu\{|b| > t\} \le C t^{-1}\|f\|_{\mathcal{Q}},\,\,\,\,\,t>0.\label{U4}
\end{gather}
\end{lemma}
\begin{proof} We shall employ an ``analytic'' cut-off function provided by Lemma \ref{outFunction}, but before that it is useful to cut $f$ into two pieces crudely (by truncation). Specifically, we write first $f = \alpha + \beta$,  where
$$
\alpha = \min\{\lambda, |f|\} \sgn f, \,\,\, \beta = f - \alpha, \,\,  \text{ where } \sgn f = \frac{f}{|f|},\,\,\,0/0=0.
$$
Clearly, $\beta$ is supported on the set $e=\{|f|>\lambda\}$ and
$$
 \mu (e) \le  \frac{\|f\|_{L^{1,\infty}}}{\lambda}.
$$
Next, it is also clear that $\|\alpha\|_{L^{1,\infty}},\,\|\beta\|_{L^{1,\infty}}\leq\|f\|_{L^{1,\infty}}$.
Now, it is easily seen that $\|\alpha\|_{L^q}\leq C \lambda^{1/p} \|f\|^{1/q}_{L^{1,\infty}}$.

For completeness, we reproduce the standard calculation leading to the last estimate. Consider the distribution function $\sigma(t) = \mu\{x: |\alpha(x)| > t \}$ for $\alpha$. Then $\sigma(t) = 0$ for $t > \lambda$ and $\sigma(t) \le \frac{\|f\|_{L^{1,\infty}}}{t}$. Now, we have
$$
\|\alpha\|^q_{L^q} = q \int\limits_0^{\infty} t^{q-1} \sigma(t)  dt \le C \|f\|_{L^{1,\infty}} \int\limits_{0}^{\lambda} t^{q-2} dt \le C \lambda^{q-1} \|f\|_{L^{1,\infty}},
$$
and the claim follows. Note that, surely, for this calculation we only need that $f\in L^{1,\infty}(\mu)$.

Now, suppose that we are given a function $f\in\mathcal{Q}$, and we want to split it is claimed in Lemma \ref{split}. Compared to the classical situations of Hardy spaces on the disk, we encounter here a tiny technical difficulty related to the possible absence of what is called ``the boundary maximum principle''. Specifically, if a function in $\mathcal{Q}$ happens to be integrable, it is not clear \textit{a priori} whether this function belongs to $\mathcal{F}^{\perp}$. To circumvent this difficulty, we first assume that $f\in\mathcal{F}^{\perp}$ (no control of the norm of $f$ in this space is assured, we shall only use the quantity $\|f\|_{\mathcal{Q}}$). If the lemma is proved for such $f$, then the general case also follows easily. Indeed, we represent an arbitrary  $f\in\mathcal{Q}$ in the form $f=\sum_{j\ge 0}f_j$ where $f_j\in\mathcal{F}^{\perp}$ for all $j$ and the norms of these functions in $L^{1,\infty}$ tend to zero very quickly. Then we split each summand $f_j$ as claimed in the lemma. The corresponding number $\lambda=\lambda_j$ changes with $j$; the $\lambda_j$ also should tend to zero quickly, but much slower than the norms of the $f_j$. In fact, some geometric rates would suffice in the two cases. Finally, we sum the splittings for $f_j$ over $j$. The adjustment of the details is left to the reader. It should also be noted that, in fact, in the sequel we shall only need the case where $f\in\mathcal{F}^{\perp}$  (and even $f\in\mathcal{F}^{\perp,q}$), see Subsection \ref{duality}.

So, in what follows we work with $f\in\mathcal{F}^{\perp}$. To begin with, we split $f$ crudely by truncation as described above: $f=\alpha+\beta$, etc. We are going to apply Lemma \ref{outFunction} to the function
$$
\varphi = \max\left\{1,\left(\frac{|f|}{\lambda}\right)^{1/\gamma}\right\},
$$
where $\gamma$ is a positive integer strictly greater than $p$. First, we show that
\begin{equation}\label{phi_one}
    \|\varphi-1\|_{L^p} \le C \|f\|_{1,\infty}^{1/p} \lambda^{-1/p}.  
\end{equation}
Indeed, the left-hand side does not exceed the quantity
\begin{equation}\label{norme}
 \left( \int\limits_{|f| > \lambda} \frac{|f|^{p/\gamma}}{\lambda^{p/\gamma}} \right)^{1/p} = \frac{C}{\lambda^{1/\gamma}} \left( \int\limits_X|f \chi_{\{|f|^{1/\gamma}>\lambda^{1/\gamma}\}}|^{p/\gamma}\right)^{1/p},
\end{equation}
where $\chi_c$ stands for the indicator function of a set $c$. Denoting by $\sigma$ the distribution function for $f \chi_{\{|f|^{1/\gamma}>\lambda^{1/\gamma}\}}$, we have
\begin{equation*}
    \sigma(t) = \begin{cases} \mu\{|f|^{1/\gamma} > \lambda^{1/\gamma}\},\quad &t\in[0,\lambda^{1/\gamma});\\
                  \mu\{|f|^{1/\gamma} > t\},\quad &t\in [\lambda^{1/\gamma},\infty).
    \end{cases}
\end{equation*}
Rewriting the right-hand side of \eqref{norme} in terms of this distribution function, we majorize it by the quantity
\begin{gather*}
\frac{C}{\lambda^{1/\gamma}} \left( \left(\, \int\limits_{\lambda^{1/\gamma}}^{\infty} t^{p-1}\,\, \mu\{x : |f(x)|^{1/\gamma} > t\} \, dt \right)^{1/p} + \left( \int\limits^{\lambda^{1/\gamma}}_{0} t^{p-1}\,\, \mu\{x : |f(x)|^{1/\gamma} > \lambda^{1/\gamma}\} \, dt \right)^{1/p} \right)\\
\le \frac{C}{\lambda^{1/\gamma}} \left( \left( \|f\|_{L^{1,\infty}} \int\limits_{\lambda^{1/\gamma}}^{\infty} t^{p-1 -\gamma} \, dt \right)^{1/p} + \left( \frac{\|f\|_{L^{1,\infty}}}{\lambda} \int\limits^{\lambda^{1/\gamma}}_{0} t^{p-1}\, dt \right)^{1/p} \right)\\
\le  \frac{C \|f\|_{1,\infty}^{1/p}}{\lambda^{1/\gamma}} \left( \lambda^{(p-\gamma)/(p\gamma)} +  \lambda^{(p/\gamma -1)/p} \right) = C \|f\|_{1,\infty}^{1/p} \lambda^{-1/p}.
\end{gather*}
In the course of the proof, we have used the inequality $p-1-\gamma < -1$, which is precisely our choice of $\gamma$.

Now, we apply Lemma \ref{outFunction} to $\varphi$ \textit{and the algebra} $\overline{\mathcal{E}}$, obtaining a function $\Phi \in\overline{\mathcal{E}}$ such that
\begin{gather}
 |\Phi| \le \min\left\{1, \left| \frac{\lambda}{f} \right|\right\},\notag\\
\label{Phi_phi}
 \|1 - \Phi\|_{L^p} \le C \|\varphi - 1 \|_{L^p}\le C \|f\|_{1,\infty}^{1/p} \lambda^{-1/p}.
\end{gather}
We show that the required decomposition of $f$ can be given simply by $f = \Phi f + (1- \Phi) f$.
To prove the due estimates for the functions $a=\Phi f$ and $b= (1- \Phi) f$, we shall use the ``crude'' decomposition $f=\alpha+\beta$ described above.

We observe that, by our assumptions about $f$, we clearly have $a,b,\in\mathcal{F}^{\perp}$ because the last space is a module over $\overline{\mathcal{E}}$. In order to ensure that $a\in\mathcal{F}^{\perp,q}$, it only suffices to prove the norm estimate $\|a\|_{L^q(\mu)}\leq C\lambda^{1/p}\|f\|_{\mathcal{Q}}^{1/q}$ (see \eqref{U1}) because $\mathcal{F}^{\perp,q}=\mathcal{F}^{\perp}\cap L^q(\mu)$ by definition. Now, recalling the notation $e=\{|f|>\lambda\}$ and the fact that $\beta$ is supported on $e$, we have
\begin{gather*}
\|a\|_{L^q} \le \| \Phi \alpha\|_{L^q} +\|\Phi \beta\|_{L^q} \le \|\alpha\|_{L^q} +
\left(\int\limits_{e} \frac{\lambda^q}{|f|^q}|f|^q d \mu \right)^{1/q} \\
\le C \lambda^{1/p} \|f\|_{\mathcal{Q}}^{1/q} + \lambda \mu(e)^{1/q}
\le C \lambda^{1/p} \|f\|_{\mathcal{Q}}^{1/q},
\end{gather*}
and \eqref{U1} is proved.

Next, we define $E=e=\{|f|>\lambda\}$, then \eqref{U3} and \eqref{U4} are clear. To verify \eqref{U2}, we write
$$
\int\limits_{X\setminus E} |b| \le  \int\limits_{X\setminus E} |1- \Phi||\alpha| \le \|\alpha\|_{L^q} \|1-\Phi\|_{L^p},
$$
and the required inequality follows from \eqref{Phi_phi} and the estimate for $\|\alpha\|_{L^q}$ obtained above.

\end{proof}

\subsection{End of the proof: duality and another cut-off\label{duality}} We pass to the proof itself of Theorem \ref{main}. Recall that we are given two algebras $A$ and $B$, two modules $C$ and $D$, and a number $p\in(1,\infty)$, and we must prove that the couple $(C_p\cap D_p, C\cap D)$ is $K$-closed in $(L^p(\mu), L^{\infty}(\mu))$. Let $q$ be the exponent conjugate to $p$. It is well known  (see \cite{P} or the survey \cite{K2}) that the question is equivalent to the $K$-closedness of the couple of annihilators of the spaces in question in the preduals. That is, our task is to show that the couple
\[
\big(\clos (C^{\perp}+D^{\perp}), \clos(C^{\perp,q}+D^{\perp,q})\big)
\]
is $K$-closed in $(L^1(\mu), L^q(\mu))$. The closures in the above display are taken in $L^1(\mu)$ and $L^q(\mu)$, respectively. Note that the spaces $C^{\perp}+D^{\perp}$ and $C^{\perp,q}+D^{\perp,q}$ are not necessarily closed themselves. However, we can forget about the operations of closure in what follows because, as is easily seen, when we verify the $K$-closedness for a couple $(U,V)$, it suffices to ensure the required decomposition only for a dense subset of $U+V$ (surely, we should also control the corresponding estimational constants in a uniform way). In our setting, even $C^{\perp,q}$ and $D^{\perp,q}$ are dense in $C^{\perp}$ and $D^{\perp}$, respectively, so we can comfortably restrict ourselves to the the solution of the following problem.

Suppose that $f\in C^{\perp,q} + D^{\perp,q}$ is represented in the form
\[
    f = g + h, \,\,\, g \in L^1,\, h \in L^q;
\]
find  $g_1\in C^{\perp} + D^{\perp}$ and $h_1\in C^{\perp,q} + D^{\perp,q}$ such that
\[
    f = g_1 + h_1, \,\,\, \|g_1\|_{C^{\perp} + D^{\perp}} \le C \|g\|_{L^1},\, \|h_1\|_{  (C^{\perp,q} + D^{\perp,q})} \le C \|h\|_{L^q}, 
\]
where $C$ depends only on $q$.

For brevity, we put $r = \|g\|_{L^1}$ and $s = \|h\|_{L^q}$. Next, denote by $\mathcal{Q}$ the closure of $C^{\perp}$ in $L^{1,\infty}(\mu)$. We remind the reader that, by the assumptions of the theorem, there is a projection that maps $L^q(\mu)$ boundedly onto $C^{\perp,q}$ and also $L^1(\mu)$ boundedly to $\mathcal{Q}$ (in particular, it is identical on $C^{\perp}$); moreover, it is assumed that $P(D^{\perp,q})\subset D^{\perp,q}$.

We apply Lemma \ref{split} to the function $Pg\in\mathcal{Q}$ (the ground algebra is $\overline{A}$ now). It should be noted that, though the quantity required for the subsequent calculations is $\|Pg\|_{L^{1,\infty}}\leq r$, in fact our choice of $f$ implies that $g\in L^q(\mu)$, consequently, also $Pg\in L^q(\mu)$. This observation will be of some use in the sequel, but now we write out the result of application of Lemma \ref{split}, in which we take $\lambda = r^{1/(1-q)} s^p$:
there exist $a \in C^{\perp,q}$, $b \in\mathcal{Q}$, and a set $E$ such that
\begin{gather}
\label{T1}  \|a\|_{C^{\perp,q}} \le C  \lambda^{1/p} \|P g\|^{1/q}_{\mathcal{Q}} \le C s, \\
\label{T2}  \int_{X\setminus E} |b| d\mu  \le C \|Pg\|_{\mathcal{Q}} \le C r, \\
\label{T3}  \mu(E) \le C \|P g\|_{\mathcal{Q}}\lambda^{-1} \le C   r^p s^{-p}, \\
\label{T4}   \mu\{|b| > t\} \le C \frac{\|P g\|_{\mathcal{Q}}}{t} \le C r t^{-1},\,\,\, t>0.
\end{gather}
Now, we define a function $u$ by the formula
\[
    u = (I - P)f = g+h-a-b-Ph.
\]
Since $I-P$ is zero on $C^{\perp,q}$ and takes $D^{\perp,q}$ into itself by assumption, we see that $u\in D^{\perp,q}$ because $f\in C^{\perp,q}+D^{\perp,q}$. Now we fix an integer $\gamma > p$ and introduce a function $\varphi$ by putting
\begin{equation*}
    \varphi = \max\left\{1,\left(\frac{(|g|+|b|)r^{1/(q-1)}}{s^p}\right)^{1/\gamma}\right\}.
\end{equation*}
We apply Lemma \ref{outFunction}, this time to $\overline{B}$ in the role of a ground algebra, to the exponent  $q$, and to $\varphi$. This yields yet another cut-off function function $\Phi\in\overline{B}$ with
\begin{equation}\label{main_theorem_Phi}
 |\Phi| \le \min\left\{1,  \frac{r^{1/(1-q)} s^{~p}}{|g|+|b|}\right\},\,\,\,
 \|1 - \Phi\|_{L^p} \le C \|\varphi - 1 \|_{L^p}.
\end{equation}
Finally, we define
\begin{equation*}
\psi = \Phi u -  h + P h  + a 
\end{equation*}
and claim that the required decomposition of $f$ is given by
\begin{equation*}
   f = g_1+h_1,\,\,\, g_1 = g - \psi, \, h_1 = h+\psi.
\end{equation*}
Now, $\Phi u\in D^{\perp,q}$, $Ph\in C^{\perp,q}$, and $a\in C^{\perp,q}$, whence we see that
$h_1=\Phi u+PH+a\in C^{\perp,q}+D^{\perp,q}$. Since $f$ also belongs to the last space, we deduce that $g_1\in C^{\perp,q}+D^{\perp,q}\subset C^{\perp}+D^{\perp}$. This means that it suffices to ensure only the norm estimates $\|g_1\|_{L^1(\mu)}\leq Cr$ and $\|h_1\|_{L^q(\mu)}\leq Cs$.

First, we estimate the quantity $\|g_1\|_{L^1(\mu)}$. By the definition of $\Phi$ and $u$, we obtain
\begin{gather}
 \|g_1\|_{L^1} = \|g-\psi\|_{L^1} \le C\|g\|_{L^1} + \|\Phi u - h + Ph+a\|_{L^1}\notag \\
    =  C\|g\|_{L^1} + \|\Phi g + \Phi h -\Phi a - \Phi b - \Phi Ph - h + Ph+a\|_{L^1} \notag \\
    \le C \|g\|_{L^1} + \|(1-\Phi)(Ph - h)\|_{L^1} + \|(1-\Phi)a\|_{L^1} + \|\Phi b\|_{L^1} \label{main_theorem_g_estimate} 
\end{gather}
We estimate separately the last  three summands in \eqref{main_theorem_g_estimate}. Using formulas \eqref{main_theorem_Phi}, \eqref{T2}, and \eqref{T3}, we arrive at the following inequality for the very last summand:
\begin{gather}
\label{main_theorem_phi_beta}    \|\Phi b\|_{L^1} = \int\limits_{E}|\Phi b| + \int\limits_{X\setminus E}|\Phi b| \le \int\limits_{E}
    \frac{|b|}{|g|+|b|}\frac{s^p}{r^{1/(q-1)}} + \int\limits_{X\setminus E}| b| \\
     \le \mu E \frac{s^p}{r^{1/(q-1)}} + C r \le Cr.\notag
\end{gather}
The remaining two summands are estimated with the help of the $\mathrm{H\ddot{o}lder}$ inequality, the continuity of $P$ in $L^q$, and relation \eqref{T1}: 
\begin{gather*}
    \|(1-\Phi)(Ph - h)\|_{L^1} \le C s \|1-\Phi\|_{L^p} ;\\
    \|(1-\Phi)a\|_{L^1} \le C \|1-\Phi\|_{L^p} \|a\|_{L^q} \le C s \|1-\Phi\|_{L^p}.
\end{gather*}
Thus, it suffices to show that
\begin{equation}\label{main_theorem_Phi_ab}
   \|1-\Phi\|_{L^p} \le C rs^{-1},
\end{equation}
By \eqref{main_theorem_Phi}, it suffices to estimate the quantity $\|\varphi - 1\|_{L^p}$. We introduce the distribution function  $\sigma$ for $\varphi \chi_{\{x:|\varphi(x)| > 1\}}$:
\begin{equation*}
    \sigma(t) = 
    \begin{cases} \mu \{|\varphi| > 1\},\quad &t\in[0,1];\\
                  \mu \{|\varphi| > t\} ,\quad &t\in (1,\infty).
    \end{cases}
\end{equation*}
Now, we have
\begin{equation}\label{main_theorem_phi_q}
\int\limits_X |\varphi - 1|^p \le \int\limits_X |\varphi|^p \chi_{\{x:|\varphi(x)| > 1\}} 
\le C \left(\int\limits_0^1{t^{p-1} \mu\{|\varphi| > 1\}} dt + \int\limits_1^{\infty} t^{p-1} \mu\{|\varphi| > t\} dt \right). 
\end{equation}
By \eqref{T4} and the Chebyshev inequality, we obtain
\begin{gather*}
    \mu\{|\varphi| > 1\} = \mu \{|g| + |b| > \lambda\} \le \mu \{|g|  > \lambda/2\} + \mu \{|b| > \lambda/2\} \le C r \lambda^{-1} = C r^p s^{-p};\\
    \mu\{|\varphi| > t\} = \mu\left\{\frac{|g| + |b|}{\lambda} > t^{\gamma}\right\} \le  \mu \{|g|  > t^{\gamma}\lambda/2\} + \mu \{|b| > t^{\gamma}\lambda/2\} \le  C r^{p}s^{-p}t^{-\gamma}.
\end{gather*}
Substituting these expressions in  \eqref{main_theorem_phi_q}, we see that
\begin{equation*}
\int\limits_{X} |\varphi - 1|^p \le  C r^p s^{-p} \left(\int\limits_0^1 t^{p-1} + \int\limits_1^{\infty} t^{p-1-\gamma}  dt \right) \le  C r^p s^{-p}. 
\end{equation*}
Thus, we have proved inequality \eqref{main_theorem_Phi_ab} and, with it, inequality \eqref{main_theorem_g_estimate}.

It remains to estimate the quantity $\|h_1\|_{L^q}$. Using the expressions for $\psi$ and $u$, we obtain
\begin{equation*}
    \|h_1\|_{L^q} =  \|h + \psi \|_{L^q} = \|\Phi u + P h +a\|_{L^q} = \|\Phi(g - b -a + h - Ph) + P h +a\|_{L^q}. 
\end{equation*}
Using \eqref{T1}  and the fact that $|\Phi| \leq 1$, we deduce that
\begin{equation}\label{main_theorem_h_Lp}
    \|h_1\|_{L^q} \le Cs + \|\Phi(g - b)\|_{L^q}.
\end{equation}
The second summand is dominated by
\begin{equation*}
  \left(\int\limits_{E} |\Phi|^q\left(|g| + |b|\right)^q d\mu \right)^{1/q} + \left(\int\limits_{X\setminus E} |\Phi|^q\left(|g| + |b|\right)^q d\mu \right)^{1/q} \le \ldots
\end{equation*}
We continue by using \eqref{main_theorem_Phi} to obtain
\begin{equation*}
\ldots  \le \lambda \mu(E)^{1/q} + \left( \int \limits_{X\setminus E} \lambda^{q-1}(|g|+|b|)\right)^{1/q}.
\end{equation*}
Next, we employ \eqref{T3} and \eqref{T2} to estimate the first and the second summand in the last expression (respectively), and then take the definition of $\lambda$ into account, to dominate the last quantity by 
\begin{equation*}
    s^p r^{1/(1-q)} r^{p/q}s^{-p/q} + C \lambda^{(q-1)/q}r^{1/q} \le C s.
\end{equation*}
This proves the required estimate for $\|h_1\|_{L^q}$ and, with it, the main theorem.

\end{document}